\documentclass{amsart}
\usepackage{graphicx,amscd}
\newtheorem{thm}{Theorem}[section]

\theoremstyle{definition}

\theoremstyle{remark}

\numberwithin{equation}{section}

\begin{document}

\title[Constructing Ultrapowers from Elementary Extensions]{Constructing Ultrapowers from Elementary Extensions of Full Clones}
\author{Joseph Van Name}
\address{}
\email{jvanname@mail.usf.edu}


\begin{abstract}
Let $A$ be an infinite set. Let $\Omega(A)$ be the algebra over $A$ where every constant is a fundamental constant
and every finitary function is a fundamental operation. 
We shall give a method of representing any algebra $\mathcal{L}$ in the variety generated by
$\Omega(A)$ as limit reduced powers and even direct limits of limit reduced powers of $\mathcal{L}$. If
the algebra $\mathcal{L}$ is elementarily equivalent to $\Omega(A)$, then this construction represents
$\Omega(A)$ as a limit ultrapower and also as direct limits of limit ultrapowers of
$\Omega(A)$. This method therefore gives a method of representing Boolean ultrapowers
and other generalizations of the ultrapower construction as limit ultrapowers and direct limits
of limit ultrapowers.
\end{abstract}
\maketitle
\section{Motivation}
For this paper, let $A$ be a fixed infinite set. If $a\in A$, then let $\hat{a}$ be a constant symbol,
and if $f:A^{n}\rightarrow A$, then let $\hat{f}$ be an $n$-ary function symbol.
Let \[\mathcal{F}=\{\hat{f}|f:A^{n}\rightarrow A\,\textrm{for some}\,n\geq 1\}\cup\{\hat{a}|a\in A\}.\] Let $\Omega(A)$ be the algebra of type $\mathcal{F}$ with universe $A$ and where $\hat{a}^{\Omega(A)}=a$ for all $a\in A$ and where
$\hat{f}^{\Omega(A)}=f$ for each function $f$ of finite arity. We shall now study the variety $V(\Omega(A))$ generated by $\Omega(A)$.

It is well known that $V(\Omega(A))=HP_{S}(\Omega(A))=HSP(\Omega(A))$. Therefore every algebra
in $V(\Omega(A))$ is isomorphic to a quotient of a subdirect power of $\Omega(A)$. We shall soon see that
the quotients of the subdirect powers of $\Omega(A)$ are simply the limit reduced powers of $\Omega(A)$.

If $I$ is a set, then we shall write $\Pi(I)$ for the lattice of partitions of $I$. If
$f:I\rightarrow X$ is a function, then we shall write $\Pi(f)$ for the partition
$\{f_{-1}(\{x\})|x\in X\}\setminus\{\emptyset\}$.

\begin{thm}
Let $I$ be a set, and let $\mathcal{B}\subseteq\Omega(A)^{I}$ be a subalgebra.
Then there is a filter $F$ on $\Pi(I)$ such that $f\in B$ if and only if
$\Pi(f)\in F$.
\end{thm}
\begin{proof}
Let $F=\{\Pi(f)|f\in\mathcal{B}\}$. We shall now show that $F$ is a filter.
Let $f,g\in\mathcal{B}$, and let $i:A^{2}\rightarrow A$ be an injective function. Then
$\hat{i}^{\mathcal{B}}(f,g):I\rightarrow A$ is a function with $\hat{i}^{\mathcal{B}}(f,g)\in B$ and
$\Pi(\hat{i}^{\mathcal{B}}(f,g))=\Pi(f)\wedge\Pi(g)$. If $f\in B$ and
$\Pi(f)\preceq P$, then there is a function $L:A\rightarrow A$ such that
$\Pi(\hat{L}^{\mathcal{B}}(f))=\Pi(L\circ f)=P$. 
Therefore $F$ is a filter. 

We now claim that $\mathcal{B}=\{f\in\Omega(A)^{I}|\Pi(f)\in F\}$. If $\Pi(f)\in F$, then
there is a function $g\in B$ with $\Pi(f)=\Pi(g)$. Therefore there is a function
$i:A\rightarrow A$ such that $f=i\circ g=\hat{i}^{\mathcal{B}}(g)$. Therefore $f\in B$.
\end{proof}

In other words, every subalgebra of $\Omega(A)^{I}$ is of the form
$\{f\in\Omega(A)^{I}|\Pi(f)\in F\}$ for some filter $F\subseteq\Pi(I)$. We shall write
$\Omega(A)^{F}$ for the algebra $\{f\in\Omega(A)^{I}|\Pi(f)\in F\}$.
One can easily show that $\{\emptyset\}\cup\bigcup F$ is a Boolean algebra. We shall now give
a one-to-one correspondence between the filters on $\{\emptyset\}\cup\bigcup F$ and the congruences on
$\Omega(A)^{F}$.

If $Z\subseteq\{\emptyset\}\cup\bigcup F$ is a filter, then let $\theta\subseteq\Omega(A)^{F}\times\Omega(A)^{F}$
be the relation where we have $(f,g)\in\theta$ if and only if $\{i\in I|f(i)=g(i)\}\in Z$. One can easily show that
$\theta$ is a congruence on $Z$. We shall let $\Omega(A)^{F}/Z$ denote the quotient algebra
$\Omega(A)^{F}/\theta$, and we shall call $\Omega(A)^{F}/Z$ a limit reduced power of $\Omega(A)$. If
$Z$ is an ultrafilter, then we shall call $\Omega(A)^{F}/Z$ a limit ultrapower of $\Omega(A)$. If $Z$ is a filter
on the set $I$, then we shall write $\Omega(A)^{I}/Z$ for $\Omega(A)^{\Pi(I)}/Z$, and we shall call
$\Omega(A)^{I}/Z$ a reduced power of $A$, and if $Z$ is an ultrafilter, then we shall simply call $\Omega(A)^{I}/Z$ an
ultrapower of $A$. The following theorem shows that every quotient of $\Omega(A)^{F}$ is a limit reduced power of $\Omega(A)$.

\begin{thm}
Let $F\subseteq\Pi(I)$ be a filter. Let $\theta$ be a congruence on $\Omega(A)^{F}$.
Then define $Z\subseteq\{\emptyset\}\cup\bigcup F$ to be the set where we have $R\in Z$ if and only if whenever
$f,g\in\Omega(A)^{F}$ and $f|_{R}=g|_{R}$, we have $(f,g)\in\theta$. Then $Z$ is a filter on $\{\emptyset\}\cup\bigcup F$.
Furthermore, we have $(f,g)\in\theta$ if and only if $\{i\in I|f(i)=g(i)\}\in Z$.
\end{thm}
\begin{proof}
We shall first show that $Z$ is a filter. If $R,S\in\{\emptyset\}\cup\bigcup F,R\subseteq S,R\in Z$, then whenever
$f|_{S}=g|_{S}$, we have $f|_{R}=g|_{R}$, so $(f,g)\in\theta$. Therefore $S\in Z$ as well. We conclude that $Z$ is an upper set.
Now assume that $R,S\in Z$. Assume that $f|_{R\cap S}=g|_{R\cap S}$. Then there is a function $h\in\Omega(A)^{F}$ where
$h|_{R}=f|_{R}$ and $h|_{S}=g|_{S}$. Therefore $(h,f)\in\theta,(h,g)\in\theta$, so $(f,g)\in\theta$.
Therefore $Z$ is a filter.

Now assume that $(f,g)\in\theta$. Then let
$R=\{i\in I|f(i)=g(i)\}$. Now let $f^{\sharp},g^{\sharp}$ be functions where $f^{\sharp}|_{R}=g^{\sharp}|_{R}$.
Let $P=\Pi(f)\wedge\Pi(g)\wedge\Pi(f^{\sharp})\wedge\Pi(g^{\sharp})$ and let $h:I\rightarrow A$ be a function such that
$\Pi(h)=P$. One can easily show that there is a function $\alpha:A^{2}\rightarrow A$ such that $\alpha(h(i),f(i))=f^{\sharp}(i)$ for $i\in I$ and
$\alpha(h(i),g(i))=g^{\sharp}(i)$ for $i\in I$. In other words, there is a function $\alpha$ where
$\hat{\alpha}^{\Omega(A)^{F}}(h,f)=f^{\sharp}$ and $\hat{\alpha}^{\Omega(A)^{F}}(h,g)=g^{\sharp}$.
Therefore since $(f,g)\in\theta$, we have $(f^{\sharp},g^{\sharp})\in\theta$ as well. Therefore $R\in Z$.
Similarly, if $\{i\in I|f(i)=g(i)\}\in Z$, then clearly $(f,g)\in\theta$.
We conclude that $(f,g)\in\theta$ if and only if $\{i\in I|f(i)=g(i)\}\in Z$.
\end{proof}

It is now clear that the elements of the variety $V(\Omega(A))$ are simply the algebras isomorphic to the
limit reduced powers of $\Omega(A)$. We also conclude that the lattice of congruences on $\Omega(A)^{F}$ is isomorphic
to the lattice of filters on the Boolean algebra $\{\emptyset\}\cup\bigcup F$. Furthermore, if $\Omega(A)^{F}/Z$ is a limit reduced
power, then the lattice of congruences on $\Omega(A)^{F}/Z$ is isomorphic to the lattice of filters on the
Boolean algebra $(\{\emptyset\}\cup\bigcup F)/Z$.

Let $\mathcal{L}\in V(\Omega(A))$. Then define a function $e:\Omega(A)\rightarrow\mathcal{L}$
by letting $e(a)=\hat{a}^{\mathcal{L}}$ for $a\in A$. One can easily show that
$e$ is the only homomorphisms from $\Omega(A)$ to $\mathcal{L}$. The following theorem shows that
every elementary extension $\mathcal{L}$ of $\Omega(A)$ is isomorphic to a limit ultrapower of
$\Omega(A)$. In the following theorem, one needs to take note that the variety $V(\Omega(A))$ is congruence
permutable(congruence permutable means that $\theta_{1}\circ\theta_{2}=\theta_{2}\circ\theta_{1}$ whenever
$\theta_{1}$ and $\theta_{2}$ are congruences in some algebra $\mathcal{L}\in V(\Omega(A))$). 
Congruence permutability follows from the limit reduced power representation of algebras or from
Mal'Cev's characterization of congruence permutable varieties \cite{B}[Sec. 2.12].

\begin{thm}
Let $\mathcal{L}\in V(\Omega(A))$ be an algebra with more than one element. Then the following are equivalent.

1. $\mathcal{L}$ is simple.

2. $\mathcal{L}$ is subdirectly irreducible.

3. $\mathcal{L}$ is directly indecomposable.

4. The mapping $e:\Omega(A)\rightarrow\mathcal{L}$ is an elementary embedding.

If $\mathcal{L}=\Omega(A)^{F}/Z$ is a limit reduced power of $\Omega(A)$, then
the above four statements are equivalent to the following statement.

5. $Z$ is an ultrafilter on the Boolean algebra $\{\emptyset\}\cup\bigcup F$.
\end{thm}
\begin{proof}
Since every algebra in $V(\Omega(A))$ is isomorphic to some limit reduced power of $A$, we may assume that
$\mathcal{L}=\Omega(A)^{F}/Z$.

$1\rightarrow 2,2\rightarrow 3$ These directions are trivial.

$5\rightarrow 4$ This is a consequence of Los's theorem for limit ultrapowers \cite{K}[Sec.\,6.4].

$5\rightarrow 1$. If $Z$ is an ultrafilter, then since $\textrm{Con}(\Omega(A)^{F}/Z)$ is isomorphic to the lattice
of congruences on $(\{\emptyset\}\cup\bigcup F)/Z$, there are only $2$ congruences on $\Omega(A)^{F}/Z$.

$4\rightarrow 3$ We shall prove this direction by contrapositive. Assume that $\mathcal{L}=\mathcal{L}_{1}\times\mathcal{L}_{2}$
where $\mathcal{L}_{1}$ and $\mathcal{L}_{2}$ are non-trivial algebras. Let $a,b\in A$ be distinct elements, and let
$i:A\rightarrow A$ be a function with $i''(A)=\{a,b\}$ and where $i(a)=a,i(b)=b$. Then $\Omega(A)$ satisfies the sentence
$\forall x(\hat{i}(x)=\hat{a}\vee\hat{i}(x)=\hat{b})$. However, we have $\hat{i}^{\mathcal{L}}(\hat{a}^{\mathcal{L}_{1}},\hat{b}^{\mathcal{L}_{2}})
=(\hat{i}^{\mathcal{L}_{1}}(\hat{a}^{\mathcal{L}_{1}}),\hat{i}^{\mathcal{L}_{2}}(\hat{b}^{\mathcal{L}_{2}}))=
(\hat{a}^{\mathcal{L}_{1}},\hat{b}^{\mathcal{L}_{2}})$, but $(\hat{a}^{\mathcal{L}_{1}},\hat{b}^{\mathcal{L}_{2}})\neq\hat{a}^{\mathcal{L}}$
and $(\hat{a}^{\mathcal{L}_{1}},\hat{b}^{\mathcal{L}_{2}})\neq\hat{b}^{\mathcal{L}}$. Therefore
$\mathcal{L}\not\models\forall x(\hat{i}(x)=\hat{a}\vee\hat{i}(x)=\hat{b})$. Therefore the mapping $e$ is not an elementary embedding.

$3\rightarrow 5$ If $Z$ is not an ultrafilter on $\{\emptyset\}\cup\bigcup F$, then since the lattice of congruences
on $\Omega(A)^{F}/Z$ is isomorphic to $\textrm{Con}((\{\emptyset\}\cup\bigcup F)/Z)$, there is a pair $\theta_{1},\theta_{2}$
of non-trivial congruences such that $\theta_{1}\cap\theta_{2}=\{(x,x)|x\in X\}$ and $\theta_{1}\vee\theta_{2}=X^{2}$.
Clearly, we have $\theta_{1}\circ\theta_{2}=\theta_{2}\circ\theta_{1}$ since variety $V(\Omega(A)^{F})/Z$ is congruence
permutable. Therefore, we have \[\Omega(A)^{F}/Z\simeq(\Omega(A)^{F}/Z)/\theta_{1}\times(\Omega(A)^{F}/Z)/\theta_{2}\]
by \cite{B}[Sec. 2.7], so $\Omega(A)^{F}/Z$ is not direct indecomposable.
\end{proof}
See \cite{K}[Sec. 6.4] for a similar but more model theoretic proof that every elementary extension of $\Omega(A)$ is a limit ultrapower of
$\Omega(A)$, and see \cite{F} for an algebraic proof of this result. We shall now represent the free algebras in $V(\Omega(A))$
as algebras of the form $\Omega(A)^{F}$. Since every algebra in $V(\Omega(A))$ can easily be represented as a quotient of a
free algebra, one can easily represent any algebra in $V(\Omega(A))$ as a quotient of $\Omega(A)^{F}$, so the algebras in $V(\Omega(A))$
are representable as limit reduced powers and limit ultrapowers of $\Omega(A)$.

If $P$ is a partition of a set $X$, then we shall write $x=y(P)$ if $x$ and $y$ are contained in the same
block of the partition $P$. Let $I$ be a set. If $i_{1},\dots,i_{n}\in I$, then let $\mathcal{P}_{i_{1},\dots,i_{n}}$ be the partition of
$A^{I}$ where $f=g(\mathcal{P}_{i_{1},\dots,i_{n}})$ if and only if
$f(i_{1})=g(i_{1}),\dots,f(i_{n})=g(i_{n})$. Clearly $\{\mathcal{P}_{i_{1},\dots,i_{n}}|n\in\mathbb{N},i_{1},\dots,i_{n}\in I\}$
is a filterbase on $\Pi(A^{I})$. We shall write $\mathcal{P}(A,I)$ for the filter generated by the filterbase
$\{\mathcal{P}_{i_{1},\dots,i_{n}}|n\in\mathbb{N},i_{1},\dots,i_{n}\in I\}$. Let $\mathbf{F}(A,I)=\Omega(A)^{\mathcal{P}(A,I)}$.

We shall now show that $\mathbf{F}(A,I)$ is a free algebra.
Let $\pi_{i}:A^{I}\rightarrow A$ be the projection function where $\pi_{i}(f)=f(i)$ for each $f:I\rightarrow A$.
Clearly $\Pi(\pi_{i})=\mathcal{P}_{i}$ since $\pi_{i}(f)=\pi_{i}(g)$ if and only if $f(i)=g(i)$ if and only if 
$f=g(\mathcal{P}_{i})$. Therefore $\pi_{i}\in\mathbf{F}(A,I)$ for all $i\in I$.

\begin{thm}
The functions $(\pi_{i})_{i\in I}$ freely generate $\mathbf{F}(A,I)$. 
\end{thm}
\begin{proof}
For each $i\in I$, we have $\Pi(\pi_{i})=\mathcal{P}_{i}$. Therefore we have
$\langle\{\pi_{i}|i\in I\}\rangle=\Omega(A)^{\mathcal{P}(A,I)}=\mathbf{F}(A,I)$, so $\{\pi_{i}|i\in I\}$
generates $\mathbf{F}(A,I)$.

We shall now show that $(\pi_{i})_{i\in I}$ freely generates $\mathbf{F}(A,I)$. It suffices to show that
whenever $f:A^{n}\rightarrow A,g:A^{m}\rightarrow A$ and $\hat{f}^{\mathbf{F}(A,I)}(\pi_{i_{1}},\dots,\pi_{i_{n}})=
\hat{g}^{\mathbf{F}(A,I)}(\pi_{j_{1}},\dots,\pi_{j_{m}})$, then the identity
$\hat{f}(x_{i_{1}},\dots,x_{i_{n}})=\hat{g}(x_{j_{1}},\dots,x_{j_{m}})$ holds. If $(a_{i})_{i\in I}\in A^{I}$, then we have
\[f(a_{i_{1}},\dots,a_{i_{n}})=f(\pi_{i_{1}}(a_{i})_{i\in I},\dots,\pi_{i_{n}}(a_{i})_{i\in I})=
\hat{f}^{\mathbf{F}(A,I)}(\pi_{i_{1}},\dots,\pi_{i_{n}})(a_{i})_{i\in I}\]
\[=\hat{g}^{\mathbf{F}(A,I)}(\pi_{j_{1}},\dots,\pi_{j_{n}})(a_{i})_{i\in I}=g(a_{j_{1}},\dots,a_{j_{m}}).\]
Therefore the identity $\hat{f}(x_{i_{1}},\dots,x_{i_{n}})=\hat{g}(x_{j_{1}},\dots,x_{j_{m}})$ holds in
the variety $V(\Omega(A))$.
\end{proof}

If $\alpha:I\rightarrow\mathcal{L}$, then let $\phi_{\alpha}:\mathbf{F}(A,I)\rightarrow\mathcal{L}$
be the unique homomorphism where we have $\phi_{\alpha}(\pi_{i})=\alpha(i)$ for $i\in I$. One can clearly see that
\[\phi_{\alpha}(\hat{f}^{\mathbf{F}(A,I)}(\pi_{i_{1}},\dots,\pi_{i_{n}}))=
\hat{f}^{\mathcal{L}}(\phi_{\alpha}(\pi_{i_{1}}),\dots,\phi_{\alpha}(\pi_{i_{n}}))=
\hat{f}^{\mathcal{L}}(\alpha(i_{1}),\dots,\alpha(i_{n})).\] Let $Z_{\alpha}$ be the filter on
$\{\emptyset\}\cup\bigcup\mathcal{P}(A,I)$ where \[\Omega(A)^{\mathcal{P}(A,I)}/Z_{\alpha}=\Omega(A)^{\mathcal{P}(A,I)}/\ker(\phi_{\alpha})\simeq\langle\alpha''(I)\rangle.\]
Clearly $Z_{\alpha}$ is an ultrafilter if and only if $\langle\alpha''(I)\rangle$ is simple.
If $\mathcal{L}$ is simple, then $Z_{\alpha}$ is always an ultrafilter for each $\alpha:I\rightarrow\mathcal{L}$. Let 
\[\iota_{\alpha}:\Omega(A)^{\mathcal{P}(A,I)}/Z_{\alpha}\rightarrow\langle\alpha''(I)\rangle\]
be the canonical isomorphism. In other words, we have $\iota_{\alpha}([\ell])=\phi_{\alpha}(\ell)$ where $[\ell]$ denotes the equivalence class of
$\ell$. Take note that if $\alpha''(I)$ generates $\mathcal{L}$, then $\iota_{\alpha}$ is an isomorphism from $\Omega(A)^{\mathcal{P}(A,I)}/Z_{\alpha}$ to $\mathcal{L}$. We therefore have a method of representing any algebra in $V(\Omega(A))$ as a limit reduced power of $\Omega(A)$. In particular, if the mapping $e:\Omega(A)\rightarrow\mathcal{L}$ is an elementary embedding, then
we can construct a limit ultrapower of $\Omega(A)$ isomorphic to $\mathcal{L}$.

If $\mathcal{L}$ is finitely generated, then one can easily show that $\mathcal{L}$ is generated by a single element.
Furthermore, if $\alpha:\{1,\dots,n\}\rightarrow\mathcal{L}$ is a function such that $\alpha(1),\dots,\alpha(n)$
generates $\mathcal{L}$, then since $\iota_{\alpha}:\Omega(A)^{A^{n}}/Z_{\alpha}=\mathbf{F}(A,\{1,\dots,n\})/Z_{\alpha}\rightarrow\mathcal{L}$
is an isomorphism, the algebra $\mathcal{L}$ is representable as a reduced power of $\Omega(A)$. In particular, if
$\mathcal{L}$ is simple and finitely generated, then $\mathcal{L}$ is representable as an ultrapower of $\Omega(A)$.
Conversely, if $|I|\leq|A|$, then every reduced power and ultrapower of $\Omega(A)$ of the form $\Omega(A)^{I}/Z$ is
finitely generated.

In the remainder of this paper, we shall discuss a method of representing every algebra $\mathcal{L}\in V(\Omega(A))$ as a direct limit
of limit reduced powers of $\Omega(A)$. By representing algebras $\mathcal{L}$ as direct limits of limit reduced powers of
$\Omega(A)$, one may be able to represent $\mathcal{L}$ as a limit reduced power besides the quotients of
the algebra $\mathbf{F}(A,I)$. Furthermore, one may also represent $\mathcal{L}$ in terms of Boolean reduced powers and
other generalizations of the reduced power and ultrapower constructions.

If $X$ is a set and $F$ is a filter on $\Pi(X)$, then the covers $F$ generate a uniformity
on $X$, so we may shall regard $(X,F)$ as a uniform space. We shall call the partitions
in the filter $F$ uniform partitions. One may refer to \cite{I} for information about uniform spaces, but no prior knowledge of
uniform spaces is necessary to finish reading this paper.

If $f:X\rightarrow Y$ is a function and $P$ is a partition of $Y$, then we shall write $[f]_{-1}(P)$ for the
partition $\{f_{-1}(R)|R\in P\}\setminus\{\emptyset\}$. One can easily show that the following properties hold.

1. $[f]_{-1}(P_{1}\wedge\dots\wedge P_{n})=[f]_{-1}(P_{1})\wedge\dots\wedge[f]_{-1}(P_{n})$, and

2. $[f\circ g]_{-1}(P)=[g]_{-1}[f]_{-1}(P)$.

If $X,Y$ are sets and $F\subseteq\Pi(X),G\subseteq\Pi(Y)$ are filters, then a function $f:X\rightarrow Y$ is said to be uniformly
continuous if whenever $P\in G$, then $[f]_{-1}(P)\in F$. If $G$ is generated by
a filterbase $\mathfrak{G}$, then $f$ is uniformly continuous if and only if whenever $P\in\mathfrak{G}$ we have
$[f]_{-1}(P)\in F$. Clearly the composition of uniformly continuous maps is uniformly continuous.

The sets $A^{I}$ shall always be given the uniformity generated by the filter $\mathcal{P}(A,I)$. Furthermore, the set
$A$ shall always have the uniformity generated by $\Pi(A)$.

\begin{thm}$\label{ProdPartUnif}$
A function $f:A^{I}\rightarrow A^{J}$ is uniformly continuous if and only if for each projection
$\pi_{j}:A^{J}\rightarrow A$, we have $\pi_{j}\circ f$ be uniformly continuous.
\end{thm}
\begin{proof}
$\rightarrow$ The projections $\pi_{j}$ are all uniformly continuous, so the mappings $\pi_{j}\circ f$ are uniformly
continuous as well being the composition of two uniformly continuous functions.

$\leftarrow$ Assume that each $\pi_{j}\circ f$ is uniformly continuous.
If $j_{1},\dots,j_{n}\in J$ are distinct elements, then we have $\mathcal{P}_{j_{1},\dots,j_{n}}=\mathcal{P}_{j_{1}}\wedge\dots\wedge
\mathcal{P}_{j_{n}}$. However, if $P=\{\{a\}|a\in A\}$, then we have \[\mathcal{P}_{j_{1}}=[\pi_{j_{1}}]_{-1}(P),\dots,\mathcal{P}_{j_{n}}=[\pi_{j_{n}}]_{-1}(P).\]
Therefore 
\[[f]_{-1}(\mathcal{P}_{j_{1},\dots,j_{n}})=[f]_{-1}(\mathcal{P}_{j_{1}}\wedge\dots\wedge\mathcal{P}_{j_{n}})
=[f]_{-1}(\mathcal{P}_{j_{1}})\wedge\dots\wedge[f]_{-1}(\mathcal{P}_{j_{n}})\]
\[=[f]_{-1}[\pi_{j_{1}}]_{-1}(P)\wedge\dots\wedge[f]_{-1}[\pi_{j_{n}}]_{-1}(P)
=[\pi_{j_{1}}\circ f]_{-1}(P)\wedge\dots\wedge[\pi_{j_{n}}\circ f]_{-1}(P)\]
is a uniform partition since each $\pi_{j}\circ f$ is uniformly continuous.
\end{proof}
Let $f\in\mathbf{F}(A,I)$, and let $\mathcal{L}\in V(\Omega(A))$. Then let
$\overline{f}^{\mathcal{L}}:\mathcal{L}^{I}\rightarrow\mathcal{L}$ be the mapping defined as follows.
If $f=\hat{g}^{\mathbf{F}(A,I)}(\pi_{i_{1}},\dots,\pi_{i_{n}})$, then let
$\overline{f}^{\mathcal{L}}((\ell_{i})_{i\in I})=\hat{g}^{\mathcal{L}}(\ell_{i_{1}},\dots,\ell_{i_{n}})$.
We now show that $\overline{f}^{\mathcal{L}}$ is well defined. Assume that
$f=\hat{g}^{\mathbf{F}(A,I)}(\pi_{i_{1}},\dots,\pi_{i_{n}})=\hat{h}^{\mathbf{F}(A,I)}(\pi_{j_{1}},\dots,\pi_{j_{m}})$.
Then since $(\pi_{i})_{i\in I}$ freely generates $\mathbf{F}(A,I)$, the identity
$\hat{f}(x_{i_{1}},\dots,x_{i_{n}})=\hat{g}(x_{j_{1}},\dots,x_{j_{n}})$ holds in the variety $V(\Omega(A))$, so
$\hat{g}^{\mathcal{L}}(\ell_{i_{1}},\dots,\ell_{i_{n}})=\hat{h}^{\mathcal{L}}(\ell_{j_{1}},\dots,\ell_{j_{m}})$.
Therefore, the mapping $\overline{f}^{\mathcal{L}}$ is well defined.
If $f:A^{I}\rightarrow A^{J}$ is uniformly continuous, then let
$\overline{f}^{\mathcal{L}}:\mathcal{L}^{I}\rightarrow\mathcal{L}^{J}$ be the mapping where
$\overline{f}^{\mathcal{L}}(\alpha)=(\overline{\pi_{j}\circ f}^{\mathcal{L}}(\alpha))_{j\in J}$.

\begin{thm}
Let $f:A^{I}\rightarrow A^{J},g:A^{J}\rightarrow A^{K}$ be uniformly continuous. Then
$\overline{g}^{\mathcal{L}}\circ\overline{f}^{\mathcal{L}}=\overline{g\circ f}^{\mathcal{L}}$.
\end{thm}
\begin{proof}
Assume $k\in K$. Then $\pi_{k}\circ g:A^{J}\rightarrow A$ is uniformly continuous. Therefore there
are $j_{1},\dots,j_{n}\in J$ and some $r:A^{n}\rightarrow A$ such that
$\pi_{k}\circ g=\hat{r}^{\mathbf{F}(A,J)}(\pi_{j_{1}},\dots,\pi_{j_{n}})$. Furthermore, since
$\pi_{j_{1}}\circ f,\dots,\pi_{j_{n}}\circ f:A^{I}\rightarrow A$ are uniformly continuous, there are indices
$i_{1},\dots,i_{m}\in I$ and mappings $s_{1},\dots,s_{n}:A^{m}\rightarrow A$ such that
\[\pi_{j_{1}}\circ f=\hat{s_{1}}^{\mathbf{F}(A,I)}(\pi_{i_{1}},\dots,\pi_{i_{m}}),\dots,
\pi_{j_{n}}\circ f=\hat{s_{n}}^{\mathbf{F}(A,I)}(\pi_{i_{1}},\dots,\pi_{i_{m}}).\]
Now let $t:A^{m}\rightarrow A$ be the mapping where
\[t(a_{1},\dots,a_{m})=r(s_{1}(a_{i_{1}},\dots,a_{i_{m}}),\dots,s_{n}(a_{i_{1}},\dots,a_{i_{m}})).\]
Then we have \[\pi_{k}\circ g\circ f(a_{i})_{i\in I}=\hat{r}^{\mathbf{F}(A,I)}(\pi_{j_{1}},\dots,\pi_{j_{n}})(f(a_{i})_{i\in I})
=r(\pi_{j_{1}}\circ f(a_{i})_{i\in I},\dots,\pi_{j_{n}}\circ f(a_{i})_{i\in I})\]
\[=r(\hat{s_{1}}^{\mathbf{F}(A,I)}(\pi_{i_{1}},\dots,\pi_{i_{m}})(a_{i})_{i\in I},\dots,
\hat{s_{n}}^{\mathbf{F}(A,I)}(\pi_{i_{1}},\dots,\pi_{i_{m}})(a_{i})_{i\in I})\]
\[=r(s_{1}(a_{i_{1}},\dots,a_{i_{m}}),\dots,s_{n}(a_{i_{1}},\dots,a_{i_{m}}))
=t(a_{i_{1}},\dots,a_{i_{m}})=\hat{t}^{\mathbf{F}(A,I)}(\pi_{i_{1}},\dots,\pi_{i_{m}})(a_{i})_{i\in I},\]
so $\pi_{k}\circ g\circ f=\hat{t}^{\mathbf{F}(A,I)}(\pi_{i_{1}},\dots,\pi_{i_{m}})$.

We also have \[\pi_{k}\circ\overline{g}^{\mathcal{L}}\circ\overline{f}^{\mathcal{L}}(\ell_{i})_{i\in I}
=\pi_{k}\circ\overline{g}^{\mathcal{L}}(\overline{\pi_{j}\circ f}^{\mathcal{L}}(\ell_{i})_{i\in I})_{j\in J}\]
\[=\overline{\pi_{k}\circ g}^{\mathcal{L}}(\overline{\pi_{j}\circ f}^{\mathcal{L}}(\ell_{i})_{i\in I})_{j\in J}
=\hat{r}^{\mathcal{L}}(\overline{\pi_{j_{1}}\circ f}^{\mathcal{L}}(\ell_{i})_{i\in I},\dots,
\overline{\pi_{j_{n}}\circ f}^{\mathcal{L}}(\ell_{i})_{i\in I})\]
\[=\hat{r}^{\mathcal{L}}(\hat{s_{1}}^{\mathcal{L}}(\ell_{i_{1}},\dots,\ell_{i_{m}}),\dots,\hat{s_{n}}^{\mathcal{L}}(\ell_{i_{1}},\dots,\ell_{i_{m}}))=\hat{t}(\ell_{i_{1}},\dots,\ell_{i_{m}})\]
\[=\overline{\pi_{k}\circ g\circ f}^{\mathcal{L}}(\ell_{i})_{i\in I}=\pi_{k}\circ\overline{g\circ f}^{\mathcal{L}}(\ell_{i})_{i\in I}.\]

 Therefore, we have
$\pi_{k}\circ\overline{g}^{\mathcal{L}}\circ\overline{f}^{\mathcal{L}}=\pi_{k}\circ\overline{g\circ f}^{\mathcal{L}}$
for all $k$. We conclude that $\overline{g}^{\mathcal{L}}\circ\overline{f}^{\mathcal{L}}=\overline{g\circ f}^{\mathcal{L}}$.
\end{proof}
If $f:A^{I}\rightarrow A^{J}$ is uniformly continuous, then define a mapping
$f^{\star}:\mathbf{F}(A,J)\rightarrow\mathbf{F}(A,I)$ by $f^{\star}(g)=g\circ f$ whenever $g\in\mathbf{F}(A,J).$
\begin{thm}
If $f:A^{I}\rightarrow A^{J}$, $\alpha:I\rightarrow\mathcal{L},\beta:J\rightarrow\mathcal{L}$ and
$\beta=\overline{f}^{\mathcal{L}}(\alpha)$, then 

1. $\phi_{\beta}=\phi_{\alpha}f^{\star}$, and

2. Whenever $R\in\{\emptyset\}\cup\bigcup\mathcal{P}(A,J)$, we have $R\in Z_{\beta}$ if and only if
$f_{-1}(R)\in Z_{\alpha}$
\end{thm}
\begin{proof}
1. Let $\ell\in\mathbf{F}(A,J)$. Then there are $j_{1},\dots,j_{n}\in J$ along with some map
$r:A^{n}\rightarrow A$ such that $\ell=\hat{r}^{\mathbf{F}(A,J)}(\pi_{j_{1}},\dots,\pi_{j_{n}})$.
Since $\pi_{j_{1}}\circ f,\dots,\pi_{j_{n}}\circ f\in F(A,I)$ there are $i_{1},\dots,i_{m}\in I$ and functions
$s_{1},\dots,s_{n}:A^{m}\rightarrow A$ where 
\[\pi_{j_{1}}\circ f=\hat{s_{1}}^{\mathbf{F}(A,I)}(\pi_{i_{1}},\dots,\pi_{i_{m}})
,\dots,\pi_{j_{n}}\circ f=\hat{s_{n}}^{\mathbf{F}(A,I)}(\pi_{i_{1}},\dots,\pi_{i_{m}}).\]

 Let $t:A^{m}\rightarrow A$
be the function where \[t(a_{1},\dots,a_{m})=r(s_{1}(a_{1},\dots,a_{m}),\dots,s_{n}(a_{1},\dots,a_{m})).\]

Then we have \[\phi_{\beta}(\ell)=\phi_{\beta}(\hat{r}^{\mathbf{F}(A,J)}(\pi_{j_{1}},\dots,\pi_{j_{n}}))\]
\[=\hat{r}^{\mathcal{L}}(\beta(j_{1}),\dots,\beta(j_{n}))
=\hat{r}^{\mathcal{L}}(\overline{f}^{\mathcal{L}}(\alpha)(j_{1}),\dots,\overline{f}^{\mathcal{L}}(\alpha)(j_{n}))\]
\[=\hat{r}^{\mathcal{L}}(\overline{\pi_{j_{1}}\circ f}^{\mathcal{L}}(\alpha),\dots,
\overline{\pi_{j_{n}}\circ f}^{\mathcal{L}}(\alpha))\]
\[=\hat{r}^{\mathcal{L}}(\hat{s_{1}}^{\mathcal{L}}(\alpha(i_{1}),\dots,\alpha(i_{m})),\dots,
\hat{s_{n}}^{\mathcal{L}}(\alpha(i_{1}),\dots,\alpha(i_{m})))\]
\[=\hat{t}^{\mathcal{L}}(\alpha(i_{1}),\dots,\alpha(i_{m})).\]

Now assume that $(a_{i})_{i\in I}\in A^{I}$. Then 
\[f^{*}(\ell)(a_{i})_{i\in I}=\ell\circ f(a_{i})_{i\in I}=\hat{r}^{\mathbf{F}(A,J)}(\pi_{j_{1}},\dots,\pi_{j_{n}})(f(a_{i})_{i\in I})\]
\[=r(\pi_{j_{1}}\circ f(a_{i})_{i\in I},\dots,\pi_{j_{n}}\circ f(a_{i})_{i\in I})\]
\[=r(\hat{s_{1}}^{\mathbf{F}(A,I)}(\pi_{i_{1}},\dots,\pi_{i_{m}})(a_{i})_{i\in I},\dots,
\hat{s_{n}}^{\mathbf{F}(A,I)}(\pi_{i_{1}},\dots,\pi_{i_{m}})(a_{i})_{i\in I})\]
\[=r(s_{1}(a_{i_{1}},\dots,a_{i_{m}}),\dots,s_{n}(a_{i_{1}},\dots,a_{i_{m}}))=t(a_{i_{1}},\dots,a_{i_{m}})\]
\[=\hat{t}^{\mathbf{F}(A,I)}(\pi_{i_{1}},\dots\pi_{i_{m}})(a_{i})_{i\in I}.\]

We conclude that \[\phi_{a}(f^{*}(\ell))=\phi_{a}(\hat{t}^{\mathbf{F}(A,I)}(\pi_{i_{1}},\dots\pi_{i_{m}}))=
\hat{t}^{\mathcal{L}}(\alpha(i_{1}),\dots,\alpha(i_{m}))=\phi_{\beta}(\ell).\]

2. Let $\ell_{1},\ell_{2}:A^{J}\rightarrow A$ be two functions such that $\{\mathbf{a}\in A^{J}|\ell_{1}(\mathbf{a})=\ell_{2}(\mathbf{a})\}=R$. Then $\mathbf{a}\in f_{-1}(R)$ if and only if
$f(\mathbf{a})\in R$ if and only if $\ell_{1}\circ f(\mathbf{a})=\ell_{2}\circ f(\mathbf{a})$. Thus
\[f_{-1}(R)=\{\mathbf{a}\in A^{I}|\ell_{1}\circ f(\mathbf{a})=\ell_{2}\circ f(\mathbf{a})\}
=\{\mathbf{a}\in A^{I}|f^{\star}(\ell_{1})(\mathbf{a})=f^{\star}(\ell_{2})(\mathbf{a})\}.\]
Therefore, we have $R\in Z_{\beta}$ if and only if $\phi_{\beta}(\ell_{1})=\phi_{\beta}(\ell_{2})$
if and only if $\phi_{\alpha}f^{*}(\ell_{1})=\phi_{\alpha}f^{*}(\ell_{2})$ if and only if $(f^{\star}(\ell_{1}),f^{\star}(\ell_{2}))\in\ker(\phi_{\alpha})$ if and only if
\[f_{-1}(R)=\{\mathbf{a}\in A^{I}|f^{\star}(\ell_{1})(\mathbf{a})=f^{\star}(\ell_{2})(\mathbf{a})\}\in Z_{\alpha}.\]
\end{proof}

In particular, for $\ell,\mathfrak{m}\in\mathbf{F}(A,J)$, if $\{\mathbf{a}\in A^{J}|\ell(\mathbf{a})=\mathfrak{m}(\mathbf{a})\}\in Z_{\beta}$, then
\[\{\mathbf{a}\in A^{I}|\ell(f(\mathbf{a}))=\mathfrak{m}(f(\mathbf{a}))\}=
f_{-1}(\{\mathbf{a}\in A^{J}|\ell(\mathbf{a})=\mathfrak{m}(\mathbf{a})\})\in Z_{\alpha}.\]
Therefore define a mapping $f^{\beta,\alpha}:\mathbf{F}(A,J)/Z_{\beta}\rightarrow\mathbf{F}(A,I)/Z_{\alpha}$ by
$f^{\beta,\alpha}([\ell])=[\ell\circ f]=[f^{\star}(\ell)]$.
Let $\iota_{\beta,\alpha}:\langle \beta''(J)\rangle\rightarrow\langle \alpha''(I)\rangle$ be the inclusion mapping.

\begin{thm}
We have $\iota_{\beta,\alpha}\iota_{\beta}=\iota_{\alpha}f^{\beta,\alpha}$.

\[\begin{CD}
\mathbf{F}(A,J)/Z_{\beta}    								@>f^{\beta,\alpha}>>    				\mathbf{F}(A,I)/Z_{\alpha}\\
@VV\iota_{\beta} V														    @VV\iota_{\alpha} V\\
\langle\beta''(J)\rangle	 @>\iota_{\beta,\alpha}>>   \langle\alpha''(I)\rangle
\end{CD}\]
\end{thm}
\begin{proof}
Let $\ell\in\mathbf{F}(A,J)$. Then $\iota_{\beta,\alpha}\iota_{\beta}[\ell]
=\iota_{\beta}[\ell]=\phi_{\beta}(\ell)=\phi_{\alpha}(f^{\star}(\ell))=\iota_{\alpha}[f^{\star}(\ell)]=
\iota_{\alpha}f^{\beta,\alpha}[\ell]$. Therefore $\iota_{\beta,\alpha}\iota_{\beta}=\iota_{\alpha}f^{\beta,\alpha}$.
\end{proof}
Since $\iota_{\alpha}$ is bijective, we have $f^{\beta,\alpha}
=\iota_{\alpha}^{-1}\iota_{\beta,\alpha}\iota_{\beta}$, and in particular, the function
$f^{\beta,\alpha}$ does not depend on $f$. We shall therefore write $E^{\beta,\alpha}$ for the mapping
$f^{\beta,\alpha}=\iota_{\alpha}^{-1}\iota_{\beta,\alpha}\iota_{\beta}$.

If $\alpha:I\rightarrow\mathcal{L},\beta:J\rightarrow\mathcal{L}$, then we shall write
$\beta\leq\alpha$ if $\langle\beta''(J)\rangle\subseteq\langle\alpha''(I)\rangle$. Clearly, the relation $\leq$
is a preordering on the class of all functions with range $\mathcal{L}$.
One can clearly see that $\beta\leq\alpha$ if and only if for each $j\in J$ there is a $f:A^{n}\rightarrow A$ and
$i_{1},\dots,i_{n}\in I$ such that $\beta(j)=\hat{f}^{\mathcal{L}}(\alpha(i_{1}),\dots,\alpha(i_{n}))$.
Furthermore, using theorem $\ref{ProdPartUnif}$,
one may show that $\beta\leq\alpha$ if and only if there is a uniformly continuous mapping
$f:A^{I}\rightarrow A^{J}$ such that $\beta=\overline{f}^{\mathcal{L}}(\alpha)$.

Assume $D$ is a directed set, and for $d\in D$ there is a set $I_{d}$ and a function
$\alpha_{d}:I_{d}\rightarrow\mathcal{L}$, and also assume $\alpha_{d}\leq\alpha_{e}$ whenever $d\leq e$, and
that $\mathcal{L}=\bigcup_{d\in D}\langle\alpha_{d}''(I_{d})\rangle$. Then we have
$\mathcal{L}=^{Lim}_{\longrightarrow}(\langle\alpha_{d}''(I_{d})\rangle,\iota_{\alpha_{d},\alpha_{e}})_{d\leq e}.$ However,
since each $\iota_{\alpha_{d}}:\Omega(A)^{\mathcal{P}(A,I_{d})}/Z_{d}\rightarrow\langle\alpha_{d}''(I_{d})\rangle$ is an isomorphism
and $\iota_{\alpha_{d},\alpha_{e}}\iota_{\alpha_{d}}=\iota_{\alpha_{e}}E^{\alpha_{d},\alpha_{e}}$, we have
\[\mathcal{L}\simeq^{Lim}_{\longrightarrow}(\mathbf{F}(A,I_{d})/Z_{d},E_{d,e})_{d\leq e}.\]

In fact, if we can find a directed system of mappings
$(f_{d,e})_{d,e\in D,d\leq e}$ such that $\overline{f}_{d,e}(\alpha_{e})=\alpha_{d}$ whenever $d\leq e$,
then we can represent $\mathcal{L}$ as a generalization of the Boolean reduced power construction called a Boolean partition algebra reduced power.

\bibliographystyle{amsplain}

\begin{thebibliography}{HD}
\bibitem{B}
Burris, Stanley, and H. P. Sankappanavar. A Course in Universal Algebra. New York: Springer-Verlag, 1981. www.math.uwaterloo.ca/~snburris.

\bibitem{F}
A simple algebraic characterization of nonstandard extensions.
Marco Forti. Proc. Amer. Math. Soc. 140 (2012), 2903-2912 

\bibitem{I}
Isbell, J. R. Uniform Spaces,. Providence: American Mathematical Society, 1964. 

\bibitem{K}
Chang, Chen Chung, and H. Jerome Keisler. Model Theory. Amsterdam: North-Holland Pub., 1973.

\end{thebibliography}

\end{document}